\documentclass[a4paper,11pt]{amsart}

\usepackage{graphicx, amssymb, amscd}  
\usepackage{mathrsfs}   
\usepackage[numbers, sort&compress]{natbib}
\usepackage{url,hyperref} 
\usepackage[inner=2.5cm,outer=2.5cm,top=2.5cm,bottom=2.5cm,includehead,includefoot]{geometry}
\hypersetup{citecolor=red, linkcolor=blue, colorlinks=true}

\newtheorem{theorem}{Theorem}[section]
\newtheorem{proposition} [theorem]{Proposition}
\newtheorem{corollary} [theorem]{Corollary}
 \newtheorem{lemma} [theorem]{Lemma}

\theoremstyle{definition}
 \newtheorem{example}[theorem]{Example}

\renewcommand\leq{\leqslant} 
\renewcommand\geq{\geqslant}

\renewcommand\mod{\bmod} 
\newcommand{\wt}[1]{|#1|}
\DeclareMathOperator{\fix}{fix}
\DeclareMathOperator{\Aut}{Aut}

\title[Locally triangular graphs and normal quotients of the $n$-cube]{Locally triangular graphs and normal \\quotients of the $n$-cube}
\author{Joanna B. Fawcett}

\address{
Centre for the Mathematics of Symmetry and Computation, The University of Western Australia, 35 Stirling Highway, Crawley, WA 6009, Australia.\newline
Email: \texttt{joanna.fawcett@uwa.edu.au}
}

\thanks{The author was supported by the Australian Research Council  Discovery Project grant DP130100106.}

\begin{document}

\begin{abstract}
 For an integer $n\geq 2$, the triangular graph has  vertex set the $2$-subsets of $\{1,\ldots,n\}$ and edge set the pairs of $2$-subsets intersecting at one point. Such graphs are known to be halved graphs of bipartite rectagraphs, which are connected triangle-free graphs in which every $2$-path lies in a unique quadrangle.   We refine this result and provide a characterisation of connected locally triangular graphs as halved graphs of normal quotients of $n$-cubes.  To do so, we study a parameter that generalises the concept of minimum distance for a binary linear code to arbitrary automorphism groups of the $n$-cube.
 \end{abstract}

\maketitle

\vspace{-0.6cm}

\section{Introduction}

For an integer $n\geq 2$, the \textit{triangular graph} $T_n$ has vertex set the $2$-subsets of $\{1,\ldots,n\}$ and edge set the pairs of $2$-subsets intersecting at one point. A finite simple undirected graph $\Gamma$ is \textit{locally triangular} (respectively \textit{locally $T_n$})  if for every vertex $u\in V\Gamma$, the graph induced by the neighbourhood  $\Gamma(u)$  is isomorphic to a triangular graph (respectively  $T_n$). 

 Various classifications of locally triangular graphs with symmetry exist, including those that are strongly regular \cite{Mak2001},  1-homogeneous    \cite[Theorem 4.4]{JurKoo2003}, and locally rank 3  \cite{BamDevFawPra2015}. These characterisations only admit a few families of graphs, but there are many more examples of locally triangular graphs; for instance, any connected component of the distance 2 graph of a coset graph of a binary linear code  with minimum distance at least $7$ is locally triangular.

Locally triangular graphs are closely related to  \textit{rectagraphs}, which are connected triangle-free graphs with the property that every path of length $2$ lies in a unique quadrangle. Specifically,   a halved graph of a bipartite rectagraph with $c_3=3$ is locally triangular, and conversely,  every connected locally triangular graph is a halved graph of a bipartite rectagraph  with   $c_3=3$ by \cite[Proposition 4.3.9]{BroCohNeu1989}. 
Rectagraphs were first named by Perkel \cite{Per1977} and have been studied by various authors including \cite{Neu1982,Bro2006,BamDevFawPra2015}. Bipartite
rectagraphs also have links to geometry, for such graphs are the incidence graphs of
semibiplanes.

A large family of rectagraphs are quotients of the $n$-cube $Q_n$ by  \cite[Proposition 4.3.6]{BroCohNeu1989}. For a graph $ \Gamma$ and a partition $\mathcal{B}$ of $V \Gamma$, the \textit{quotient graph} $ \Gamma_\mathcal{B}$ is  the simple graph with vertex set $\mathcal{B}$, where distinct $B_1$ and $B_2$ in $\mathcal{B}$ are adjacent whenever there exist $x_1\in B_1$ and $x_2\in B_2$ such that $x_1$ and $x_2$ are adjacent in $ \Gamma$. If $\mathcal{B}$ is the set of orbits of $K\leq \Aut( \Gamma)$, then  $ \Gamma_{\mathcal{B}}$ is a \textit{normal quotient} of $ \Gamma$, and we write $ \Gamma_K$ for $ \Gamma_{\mathcal{B}}$. Normal quotients are particularly well-behaved, for we retain some control of their automorphism groups, valencies and  local actions.

Our first theorem improves the known characterisation of connected locally triangular graphs as halved graphs of bipartite rectagraphs. Note that every connected locally triangular graph is locally $T_n$ for some $n$ by  \cite[Proposition 4.3.9]{BroCohNeu1989}, so it suffices to consider connected locally $T_n$ graphs. Precise definitions for terms in the following results, including the parameter $d_K$, are given in \S \ref{s: prelim} and \S\ref{s: par}.

\begin{theorem}
\label{main even}
Let $\Gamma$ be a graph. Let $n\geq 2$.
The following are equivalent.
\begin{itemize}
\item[(i)] $\Gamma$ is  a connected locally $T_n$ graph.
\item[(ii)] $\Gamma$ is a halved graph of  $(Q_n)_K$ for some  $K\leq \Aut(Q_n)$ such that $K$ is even and  $d_K\geq 7$.
\end{itemize}
Moreover, the group $K$ in (ii) is  unique up to conjugacy in $\Aut(Q_n)$. 
\end{theorem}

\nopagebreak

In fact,  $K$ acts semiregularly on $VQ_n$ and is therefore a $2$-group. 
The even subgroups of $\Aut(Q_n)$ are essentially those groups $K$ for which $(Q_n)_K$ is bipartite (cf. Lemma \ref{even}), and the  parameter $d_K$ extends the concept of minimum distance for a binary linear code to arbitrary subgroups of $\Aut(Q_n)$. Note that either $d_K\leq n$ or $d_K=\infty$.

In order to prove Theorem \ref{main even}, we first  explore the relationship between the parameter $d_K$ and the normal quotient $(Q_n)_K$. It is well known that the $n$-cube $Q_n$ ($n\geq 1$) is a regular graph of valency $n$ with $a_{i-1}=0$ and $c_i=i$ for all $i$. In our next result, we see that the parameter $d_K$ measures the extent to which  $(Q_n)_K$ locally approximates the $n$-cube.

\begin{theorem}
\label{class. dist.}
Let $K\leq \Aut(Q_n)$ and $\ell$ a positive integer. The following are equivalent.
\begin{itemize}
\item[(i)]   $(Q_n)_K$ is  regular   of valency $n$ with  $a_{i-1}=0$ and $c_i=i$ for $1\leq i\leq \ell$.
\item[(ii)]  $d_K\geq 2\ell +1$.
\end{itemize}
\end{theorem}

Rectagraphs are precisely the connected graphs with $a_1=0$ and $c_2=2$, and they are always  regular by \cite[Proposition 1.1.2]{BroCohNeu1989}.
As a corollary of Theorem \ref{class. dist.}, \cite[Proposition 4.3.6]{BroCohNeu1989} and \cite[Proposition 3.4]{BamDevFawPra2015},
we obtain a characterisation of a large family of rectagraphs as normal quotients of $n$-cubes, from which Theorem \ref{main even}  quickly follows. 

\begin{corollary}
\label{main rect}
Let $\Pi$ be a graph. The following are equivalent.
\begin{itemize}
\item[(i)] $\Pi$ is  a rectagraph of valency $n$ with $a_2=0$ and $c_3=3$.
\item[(ii)] $\Pi\simeq(Q_n)_K$ for some  $K\leq \Aut(Q_n)$ such that  $d_K\geq 7$.
\end{itemize}
\end{corollary}

We also prove the following (cf. \cite[Proposition 3.4]{BamDevFawPra2015} and \cite[Lemma 5]{Mat1991}).

\begin{theorem}
\label{conjuate simple}
Let $K,L\leq \Aut(Q_n)$ where $d_K\geq 5$. Then $(Q_n)_K$ and $(Q_n)_L$ are isomorphic if and only if $K$ and $L$ are conjugate in $\Aut(Q_n)$. In particular, $\Aut((Q_n)_K)=N_{\Aut(Q_n)}(K)/K$.
\end{theorem}

As a consequence of Theorem \ref{main even}, we can determine the automorphism group of a connected locally $T_n$ graph. In the following,  $E_n$ denotes the set of vectors in $\mathbb{F}_2^n$ with even weight, and $E_n:S_n$ is an index 2 subgroup of $\Aut(Q_n)=\mathbb{F}_2\wr S_n$.

\begin{theorem}
 \label{main aut}
Let $\Gamma$ be a halved graph of $(Q_n)_K$ where  $K\leq \Aut(Q_n)$ is even with $d_K\geq 7$ and $n\geq 5$. Then $\Aut(\Gamma)=N_{E_n:S_n}(K)/K$.\end{theorem}

When $C$ is a binary linear code in $\mathbb{F}_2^n$, the coset graph of $C$ is the normal quotient $(Q_n)_C$, and if $C$ has minimum distance at least $5$, then $(Q_n)_C$ is a rectagraph. Moreover, as noted earlier, if $C$ has minimum distance at least $7$, then any connected component of the distance $2$ graph of $(Q_n)_C$ is locally $T_n$. Every such rectagraph and locally $T_n$ graph is vertex-transitive. However, in general there are groups $K\leq \Aut(Q_n)$ for which neither $(Q_n)_K$ nor a corresponding locally $T_n$ graph is vertex-transitive (cf. Examples \ref{not vt} and \ref{lt not vt}); in fact, in the examples we give, the parameter $d_K$ is nearly maximal.

Another property of a binary linear code $C$ is that when $(Q_n)_C$ is bipartite and $C$ has minimum distance at least $2$, the halved graphs of $(Q_n)_C$ are isomorphic. Indeed, this is a consequence of the vertex-transitivity of $(Q_n)_C$. However, this is not true for arbitrary subgroups of $\Aut(Q_n)$ (cf. Example \ref{exp:halved}). In light of Theorem \ref{main even}, it would be interesting to classify those $K\leq \Aut(Q_n)$ with $d_K\geq 2$ for which the halved graphs of a bipartite $(Q_n)_K$ are isomorphic; 
we prove that this holds for every such $K$ when $n$ is odd  (cf.   Corollary \ref{odd iso}).

This paper is organised as follows. In \S \ref{s: prelim}, we give some notation and  basic definitions. In \S \ref{s: par}, we first define and analyse the parameter $d_K$, and then we   prove Theorem \ref{class. dist.},  Corollary \ref{main rect} and Theorem \ref{conjuate simple}. In \S \ref{s: halved}, we consider some properties of the distance $2$ graph of $(Q_n)_K$, and in \S \ref{s: proof}, we prove  Theorems \ref{main even} and \ref{main aut}.

\section{Notation and basic definitions}
\label{s: prelim}

All groups in this paper are finite, all actions are written on the right, and all graphs  are finite and undirected with no multiple
edges or loops. Basic graph theoretical terminology may be found in \cite{BroCohNeu1989}.

For groups $G$ and $H$, we denote a semidirect product of $G$ with $H$ (with normal subgroup $G$) by $G: H$, and if  $H$ acts faithfully on $[n]:=\{1,\ldots,n\}$, then we denote the wreath product of $G$ with $H$ by $G\wr H$.  
If $G$ acts on $\Omega$
and $\omega\in\Omega$, then   $G_{\omega}$ denotes the pointwise stabiliser of $\omega$ in $G$ and 
$\omega^G$  the orbit of $G$ containing $\omega$. We say that $G$ is  \textit{transitive} if $\omega^G=\Omega$, and \textit{semiregular} if $G_\omega=1$ for all $\omega\in\Omega$. The symmetric group  on $n$ points is denoted by $S_n$.

Let $\Gamma$ be a graph. We write $V\Gamma$ for the vertex set of $\Gamma$, $E\Gamma$ for the edge
set of $\Gamma$, and $\Aut(\Gamma)$ for the automorphism group of $\Gamma$. If $\Aut(\Gamma)$ is transitive on $V\Gamma$, then $\Gamma$ is \textit{vertex-transitive}. If $X\subseteq V\Gamma$, then $[X]$ is the
subgraph of $\Gamma$ induced by $X$. The distance between $u,v\in V\Gamma$ is denoted by
$d_\Gamma(u,v)$.
For $u\in V\Gamma$ and  $i\geq 0$, we define $\Gamma_i(u):=\{v\in
V\Gamma:d_\Gamma(u,v)=i\}$ and $\Gamma(u):=\Gamma_1(u)$. 
For $u,v\in V\Gamma$ such that $d_\Gamma(u,v)=i$,  let $c_i(u,v) :=|\Gamma_{i-1}(u)\cap \Gamma(v)|$ and $a_i(u,v) :=|\Gamma_{i}(u)\cap \Gamma(v)|$.
We write $c_i$ (respectively $a_i$) whenever $c_i(u,v)$ (respectively $a_i(u,v)$)
does not depend on the choice of $u$ and $v$.
  The complete graph on $n$ vertices is
denoted by $K_n$, and the complete multipartite graph with $n$ parts of size $a$ is denoted by
$K_{n[a]}$.

The \textit{distance 2 graph} $\Gamma_2$ of a graph $\Gamma$ has vertex set $V\Gamma$, where two vertices
are adjacent whenever their distance in $\Gamma$ is 2. If $\Gamma$ is connected but not bipartite,
then $\Gamma_2$ is connected, and if $\Gamma$ is connected and bipartite, then $\Gamma_2$ has
exactly two connected components; these are called the \textit{halved graphs} of $\Gamma$. 
The \textit{bipartite double} $\Gamma.2$ of a graph $\Gamma$ has vertex set $V\Gamma\times \mathbb{F}_2$,
where vertices $(u,x)$ and $(v,y)$ are adjacent whenever $u$ and $v$ are adjacent in $\Gamma$ and
$x\neq y$.

For  graphs $\Gamma$ and $\Pi$, a surjective map $\pi:\Gamma\to\Pi$ is a \textit{covering} if $\pi$ induces a bijection from $\Gamma(x)$ onto $\Pi(x\pi)$ for all $x\in V\Gamma$.  We denote by $K^\pi$ the subgroup $\{g\in\Aut(\Gamma): g\pi=\pi\}$ of $\Aut(\Gamma)$, where $g\pi$ denotes the composition of the functions $g$ and $\pi$.

Let $\mathbb{F}_2^n$ be the vector space of $n$-tuples over the field $\mathbb{F}_2$. The
\textit{weight} $\wt{u}$ of a vector $u\in \mathbb{F}_2^n$ is the number of non-zero coordinates in
$u$, and the \textit{Hamming distance} of $u,v\in\mathbb{F}_2^n$ is the number of coordinates at
which $u$ and $v$ differ, or equivalently, $\wt{u+v}$.  For an $m$-subset $\{i_1,\ldots,i_m\}$ of $[n]$, let $e_{i_1,\ldots,i_m}$ denote the
vector of weight $m$ in $ \mathbb{F}_2^n$ whose $i_j$-th coordinate is 1 for $1\leq j\leq m$. The subspace of $\mathbb{F}_2^n$ consisting of vectors with
even weight is denoted by $E_n$. 

For $n\geq 1$, the \textit{$n$-cube} $Q_n$ is  the graph with vertex set   $\mathbb{F}_2^n$, where two vectors are adjacent whenever their Hamming distance is 1. The $n$-cube is a connected regular bipartite graph of valency $n$ with parts $E_n$ and $\mathbb{F}_2^n\setminus E_n$. Its automorphism group is $\mathbb{F}_2\wr
S_n=\mathbb{F}_2^n: S_n$, where $\mathbb{F}_2^n$ acts on $VQ_n$ by translation 
and $S_n$ acts by permuting coordinates. A  subgroup $K$ of $\Aut(Q_n)$ is   \textit{even} if $K\leq E_n: S_n$. We write elements of $\mathbb{F}_2^n: S_n$  in the form $(x,\sigma)$ where $x\in \mathbb{F}_2^n$ and $\sigma\in S_n$. A \textit{binary linear code} $C$ is a subspace of
$\mathbb{F}_2^n$, or, equivalently, an additive subgroup of $\mathbb{F}_2^n$. The \textit{minimum distance} of  $C$ is  the minimum weight of the non-zero codewords when $C\neq 0$, and  $\infty$  otherwise. The  \textit{coset graph} of $C$ is  the  normal quotient $(Q_n)_C$.

\section{The minimum distance of $K\leq \Aut(Q_n)$}
\label{s: par}

Let  $K\leq \Aut(Q_n)$.  As in \cite{Mat1991}, we define the \textit{minimum distance of $K$}, denoted by $d_K$, as follows:
$$d_K:=\left\{
\begin{array}{ll}
 \min \{ d_{Q_n}(x,x^k):x\in VQ_n,1\neq k\in K\} & \mbox{if}\ K\neq 1, \\
\infty & \mbox{otherwise}.
\end{array}
\right .
$$ 
 This definition generalises the concept of minimum distance for a binary linear code, for if $C\leq \mathbb{F}_2^n$, then   $d_{Q_n}(x,x^c)=d_{Q_n}(x,x+c)=\wt{c}$ for all $x\in VQ_n$ and $c\in C$.
 
 Observe that $d_K\geq 1$ if and only if $K$ is semiregular, which is true  if and only if $(Q_n)_K$ has order $2^n/|K|$.   Moreover,  $d_K=d_{g^{-1}Kg}$ for all $g\in \Aut(Q_n)$, and if $K\neq 1$, then $d_K\leq n$.

We wish to find a canonical set of representatives for the vertices at distance $\ell$ from vertex $x^K$ in $(Q_n)_K$. This is not always possible, but we can say the following. Note that $(Q_n)_\ell(x)=\{x+e\in VQ_n:\wt{e}=\ell\}$ for all $x\in VQ_n$ and $\ell\in [n]$.

\begin{lemma}
\label{nbd}
Let $K\leq \Aut(Q_n)$ and $\Pi:=(Q_n)_K$. Then $$\Pi_\ell(x^K)\subseteq\{(x+e)^K\in V\Pi:\wt{e}=\ell\}$$ for every  positive integer $\ell$ and $x
\in VQ_n$.
\end{lemma}

\begin{proof}
Fix $x\in VQ_n$.
  If $\Pi_1(x^K)$ is empty, then $\Pi_\ell(x^K)$ is empty for all $\ell\geq 1$, and we are done. Otherwise, let $y^K\in \Pi_1(x^K)$.  Then $y^{k_1}$ is adjacent to $x^{k_2}$ for some $k_1,k_2\in K$, so $y^{k_1k_2^{-1}}$ is adjacent to $x$. Thus $y^{k_1k_2^{-1}}=x+e_i$ for some $i$, and  $y^K=(x+e_i)^K$. 
 Suppose that $\ell>1$.  Again, if $\Pi_\ell(x^K)$ is empty, then we are done, so we assume otherwise. Let  $y^K\in \Pi_\ell(x^K)$. There exists a path $(z_0^K,z_1^K,\ldots,z_\ell^K)$ of length $\ell$ in $\Pi$ where $z_0^K=x^K$ and $z_\ell^K=y^K$. Now $z_{\ell-1}^K\in \Pi_{\ell-1}(x^K)$, so we may assume by induction that  $z_{\ell-1}=x+e_{i_1,\ldots,i_{\ell-1}}$ for some  $i_1<\cdots<i_{\ell-1}\in [n]$. Since $y^K\in \Pi(z_{\ell-1}^K)$, there exists $i_\ell\in [n]$ such that $y^K=(z_{\ell-1}+e_{i_\ell})^K$.  If $i_\ell\in \{i_1,\ldots,i_{\ell-1}\}$, then 
there is a path between $y^K$ and $x^K$ of length at most $\ell-2$, a contradiction. Thus $y^K=(x+e_{i_1,\ldots,i_{\ell-1}}+e_{i_\ell})^K$ where $\wt{e_{i_1,\ldots,i_{\ell-1}}+e_{i_\ell}}=\ell$.\end{proof}

Next we have an elementary observation concerning $d_K$.

\begin{lemma}
\label{trick}
Let $K\leq \Aut(Q_n)$. If $x^K= y^K$ for distinct $x,y\in VQ_n$, then $\wt{x+y}\geq d_K$.
\end{lemma}

\begin{proof}
If $x^K= y^K$ for distinct $x,y\in VQ_n$, then there exists $1\neq k\in K$ such that $x^k=y$, so $\wt{x+y}=d_{Q_n}(x,x^k)\geq d_K$.
\end{proof}

The following is a straightforward consequence of Lemma \ref{trick} and the definition of $d_K$.

\begin{lemma}
\label{cycle}
Let $1\neq K\leq \Aut(Q_n)$. If $d_K\geq 3$, then $(Q_n)_K$ has a cycle of length $d_K$.
\end{lemma}

By making some assumptions on $d_K$, we can improve Lemma \ref{nbd}.

\begin{lemma}
\label{nbd2}
Let $K\leq \Aut(Q_n)$ and $\Pi:=(Q_n)_K$.  If
$d_K\geq 2\ell$ for some positive integer $\ell$,  then $$\Pi_\ell(x^K)=\{(x+e)^K\in V\Pi:\wt{e}=\ell\}$$  for all $x\in VQ_n$.
\end{lemma}

\begin{proof}
Fix $x\in VQ_n$. By Lemma \ref{nbd}, $\Pi_\ell(x^K)\subseteq \{(x+e)^K\in V\Pi:\wt{e}=\ell\}$. In particular, we are done if $\ell>n$, so we assume that $\ell\leq n$, and let $y\in VQ_n$ be such that $\wt{x+y}=\ell$. Now $d_\Pi(x^K,y^K)\leq \ell$. If $d_\Pi(x^K,y^K)<\ell$, then $y^K=z^K$ for some $z\in VQ_n$ where $\wt{x+z}<\ell$ by Lemma \ref{nbd}, but $y\neq z$, so  $2\ell> \wt{x+y+x+z}= \wt{ y+z}\geq d_K$ by Lemma \ref{trick}, a contradiction. Thus $y^K\in \Pi_\ell(x^K)$.
\end{proof}

Although $\tbinom{n}{\ell}$ is  an upper bound on $|\{(x+e)^K\in V\Pi:\wt{e}=\ell\}|$, these need not be equal in general, even for $d_K\geq 2\ell$. Indeed, suppose that $d_K=2\ell$ where $\ell$ is a positive integer. There exists $y\in VQ_n$ and $1\neq k\in K$ such that $d_{Q_n}(y,y^k)=2\ell$, so there exists $x\in (Q_n)_\ell(y)\cap (Q_n)_\ell(y^k)$, but this implies there exist distinct $e,f\in VQ_n$  such that $\wt{e}=\ell=\wt{f}$ and $(x+e)^K=y^K=(x+f)^K$.

Given $K\leq \Aut(Q_n)$, there is a \textit{natural map} $\pi:Q_n\to (Q_n)_K$ defined by $x\mapsto x^K$ for all $x\in VQ_n$. The following result implies Theorem \ref{class. dist.} with $\ell=1$.

\begin{lemma}
\label{covering}
Let $K\leq \Aut(Q_n)$  and $\Pi:=(Q_n)_K$.  The following are equivalent.
\begin{itemize}
\item[(i)] The natural map $\pi:Q_n\to \Pi$  is a covering.
\item[(ii)] $\Pi$ is  regular of valency $n$.
\item[(iii)] $d_K\geq 3$.
\end{itemize}
\end{lemma}

\begin{proof}
Clearly (i) implies (ii), and (iii) implies (i) by Lemmas \ref{trick} and \ref{nbd2}, so it remains to prove (ii) implies (iii).  Suppose that $d_K\leq 2$. If $x^K=(x+e_i)^K$ for some $x\in VQ_n$ and $i\in [n]$, then  $|\Pi(x^K)|<n$, and if $x^K=(x+e_{i,j})^K$ for some $x\in VQ_n$ and distinct $i,j\in [n]$, then $|\Pi((x+e_i)^K)|<n$. Thus we may assume that $d_K=0$, in which case there exist $x\in VQ_n$ and $1\neq k\in K$ such that $x^k=x$. Write $k=(y,\sigma)$. Clearly $\sigma\neq 1$, so $\sigma$ moves some $i\in [n]$. Let $z:=x+e_i$. Then $z^k=x+e_{i^\sigma}=z+e_{i,i^\sigma}$, so $|\Pi((z+e_i)^K)|<n$.
\end{proof}

Next we consider the parameters $a_i$ and $c_i$.

\begin{lemma}
\label{a c}
Let $K\leq \Aut(Q_n)$  and $\Pi:=(Q_n)_K$.  Let $\ell$ be a positive integer.
\begin{itemize}
\item[(i)] If $d_K\geq 2\ell$, then $a_{\ell-1}=0$.
\item[(ii)] If $d_K\geq 2\ell+1$, then $c_\ell=\ell$.
\end{itemize}
\end{lemma}

\begin{proof}
(i) Suppose that $d_K\geq 2\ell$, and let $x^K,y^K,z^K\in V\Pi$ be such that $y^K\in \Pi_{\ell-1}(x^K)$ and $z^K\in \Pi_{\ell-1}(x^K)\cap \Pi(y^K)$.
 By Lemma \ref{nbd}, we may assume that $\wt{x+y}=\wt{x+z}=\ell-1$ and $z^K=(y+e_i)^K$ for some $i\in [n]$. Note that $z\neq y+e_i$. Now $2(\ell-1)+1\geq  \wt{z+y+e_i}\geq d_K$ by Lemma \ref{trick}, a contradiction. Thus $a_{\ell-1}=0$.

(ii) Suppose that $d_K\geq 2\ell+1$. Let $x^K,y^K\in V\Pi$ be such that $y^K\in \Pi_{\ell}(x^K)$. By Lemma \ref{nbd}, we may assume that  $y=x+e_{i_1,\ldots,i_\ell}$ for some $ i_1<i_2<\cdots<i_\ell\in [n] $. By Lemmas  \ref{trick} and \ref{nbd2}, for $1\leq j\leq \ell$, the  vertices $(x+e_{i_1,\ldots,i_\ell}+e_{i_j})^K$  are  pairwise distinct and lie in $\Pi_{\ell-1}(x^K)\cap \Pi(y^K)$.
 Thus $|\Pi_{\ell-1}(x^K)\cap \Pi(y^K)|\geq \ell$. Let $z^K \in \Pi_{\ell-1}(x^K)\cap \Pi(y^K)$. By Lemma \ref{nbd}, we may assume that $\wt{x+z}=\ell-1$ and  $z^K=(x+e_{i_1,\ldots,i_\ell}+e_i)^K$ for some $i\in [n]$. If $z\neq x+e_{i_1,\ldots,i_\ell}+e_i$, then $2\ell\geq \wt{z+x+e_{i_1,\ldots,i_\ell}+e_i}\geq d_K$ by Lemma \ref{trick}, a contradiction. Thus $z=x+e_{i_1,\ldots,i_\ell}+e_i$, so $i\in \{i_1,\ldots,i_\ell\}$. Hence $c_\ell=\ell$.
\end{proof}

The following result is a simple counting exercise; we prove it here for completeness.

\begin{lemma}
\label{counting}
Let  $\Pi$ be a  regular graph of valency $n$, and let $\ell$ be a positive integer. If  $a_{i-1}=0$ and $c_i=i$ for $1\leq i\leq \ell$, then
$|\Pi_\ell(u)|=\tbinom{n}{\ell}$
for all $u\in V\Pi$.
\end{lemma}

\begin{proof}
Let $u\in V\Pi$. Let $m:=\min(\ell,n)$. Now $\Pi_i(u)$ is non-empty for $0\leq i\leq m$, for if not, then there exists $0\leq j<m$ such that $\Pi_{j}(u)\neq\varnothing$  and $\Pi_{j+1}(u)=\varnothing$, but $a_{j}=0$ and $c_{j}=j$, so a vertex in $\Pi_{j}(u)$ has $n-j$ neighbours in $\Pi_{j+1}(u)$,  a contradiction.
In particular, if $\ell>n$, then $\Pi_n(u)\neq \varnothing$ and $a_{n}=0$ and $c_{n}=n$, so  $\Pi_{n+1}(u)=\varnothing$ and $|\Pi_\ell(u)|=0=\tbinom{n}{\ell}$. 

Thus we may assume that $\ell\leq n$.   Now $\Pi_{i}(u)$ is non-empty for $1\leq i\leq \ell$.  Since $a_{\ell-1}=0$ and $c_{\ell-1}=\ell-1$, every vertex in $\Pi_{\ell-1}(u)$ has $n-\ell+1$ neighbours in $\Pi_\ell(u)$. Thus the number of edges between $\Pi_{\ell-1}(u)$ and $\Pi_\ell(u)$ is $(n-\ell+1)|\Pi_{\ell-1}(u)|$. On the other hand, since $c_\ell=\ell$, the number of edges between $\Pi_{\ell-1}(u)$ and $\Pi_\ell(u)$ is $\ell |\Pi_\ell(u)|$. By induction, $|\Pi_{\ell-1}(u)|=\tbinom{n}{\ell-1}$, so $|\Pi_\ell(u)|=\tbinom{n}{\ell}$.
\end{proof}

\begin{proof}[Proof of Theorem \ref{class. dist.}]
Let $\Pi:=(Q_n)_K$. Suppose that $\Pi$ is regular of valency $n$ with $a_{i-1}=0$ and $c_i=i$ for $1\leq i\leq \ell$. Then   
\begin{equation}
\label{eqn:size}
|\Pi_i(x^K)|=\tbinom{n}{i}\tag{$\ast$}
\end{equation}
 for all $x\in VQ_n$ and $1\leq i\leq \ell$  by Lemma \ref{counting}, so 
\begin{equation}\label{eqn:neigh}
\Pi_i(x^K)=\{(x+e)^K\in V\Pi:\wt{e}=i\}\tag{\dag}\end{equation}
for all $x\in VQ_n$ and  $1\leq i\leq \ell$ by Lemma \ref{nbd}. Moreover, $d_K\geq 3$ by Lemma \ref{covering}, so $K$ is semiregular and $\Pi$ has $2^n/|K|$ vertices. If $K=1$, then $d_K=\infty$, as desired, so we may assume that $K\neq 1$. In particular, $\ell<n$, or else $\Pi$ has $2^n$ vertices by (\ref{eqn:size}), in which case $K=1$.

Let $x\in VQ_n$ and $1\neq k\in K$. Note that $x\neq x^k$. Now $d_{Q_n}(x,x^k)> \ell$ by (\ref{eqn:neigh}), so we may write $d_{Q_n}(x,x^k)=j+\ell$ for some positive integer $j$. Suppose for a contradiction that $j\leq \ell$.
 There exists $y\in VQ_n$ such that $d_{Q_n}(x,y)=j$ and $d_{Q_n}(y,x^k)=\ell$, so $y^K\in \Pi_j(x^K)\cap \Pi_\ell(x^K)$. Thus $j=\ell$. 
 Let $x_1:=y+x$ and $x_2:=y+x^k$. Now $x_1$ and $x_2$ are distinct vectors of weight $\ell$, while $(y+x_1)^K=(y+x_2)^K$, so $| \{(y+e)^K\in V\Pi:\wt{e}=\ell\}|<\tbinom{n}{\ell}$, contradicting (\ref{eqn:size}) and (\ref{eqn:neigh}). Hence $d_K\geq 2\ell+1$.

Conversely, suppose that $d_K\geq 2\ell +1$. Then $\Pi$ is  regular of valency $n$ by Lemma \ref{covering}, and   $a_{i-1}=0$ and $c_i=i$ for $1\leq i\leq \ell$ by Lemma \ref{a c}. 
 \end{proof}

The assumption that $(Q_n)_K$ has valency  $n$ cannot be removed from the statement of Theorem \ref{class. dist.}, as the following example shows.

\begin{example}
For each positive $m< n$, there exists a subgroup $K$ of $\Aut(Q_n)$ with $d_K=2$ for which $(Q_n)_K$ is  regular of valency $m$ with $a_{i-1}=0$ and $c_i=i$ for all $i\geq 1$. 
Define $K$ to be the set of vectors in $E_n$ whose first $m-1$ coordinates are $0$. Let $\pi:Q_n\to(Q_n)_K$ be the natural map.
 Viewing $Q_m$ as the subgraph of $Q_n$ induced by  $\mathbb{F}_2^m\times 0^{n-m}$, the restriction $\pi:Q_m\to (Q_n)_K$ is a graph isomorphism, and $Q_m$ has the desired properties.
\end{example}

\begin{proof}[Proof of Corollary \ref{main rect}]
Let $\Pi$ be a rectagraph of valency $n$ with $a_2=0$ and $c_3=3$. By \cite[Lemma 3.1]{BamDevFawPra2015} (cf. \cite[Lemma 4.3.5 and Proposition 4.3.6]{BroCohNeu1989}), there exists a covering $\pi:Q_n\to \Pi$.   Let $K:=K^\pi=\{g\in \Aut(Q_n):g\pi=\pi\}$. Then $\Pi\simeq (Q_n)_K$ by \cite[Proposition 3.4]{BamDevFawPra2015},  and   $d_K\geq 7$ by Theorem \ref{class. dist.}. The converse also follows from Theorem \ref{class. dist.}.
\end{proof}

For $K\leq \Aut(Q_n)$, the normaliser $N_{\Aut(Q_n)}(K)$ acts naturally on   $V(Q_n)_K$ by $$(x^K)^g:=(x^g)^K$$ for all $x\in VQ_n$ and $g\in N_{\Aut(Q_n)}(K)$. For $d_K\geq 5$, in which case $(Q_n)_K$ is a rectagraph, 
every automorphism of  $(Q_n)_K$ arises in this way  by \cite[Proposition 3.4]{BamDevFawPra2015} (cf. \cite[Lemma 5]{Mat1991}).  We generalise this result (and its proof) in order to prove that isomorphisms of normal quotients arise only from conjugate groups, thereby establishing 
 Theorem \ref{conjuate simple}.

\begin{proposition}
\label{conjugate}
Let $K,L\leq \Aut(Q_n)$. Let $\pi_K:Q_n\to(Q_n)_K$ and $\pi_L:Q_n\to (Q_n)_L$ denote the natural maps.
\begin{itemize}
\item[(i)] If  $d_K\geq 5$ and $\varphi:(Q_n)_K\to (Q_n)_L$ is a graph isomorphism, then there exists $g\in \Aut(Q_n)$ such that $g^{-1}Kg=L$ and the diagram below commutes.
\item[(ii)] If $g\in \Aut(Q_n)$ and $g^{-1}Kg=L$, then there exists a graph isomorphism $\varphi:(Q_n)_K\to (Q_n)_L$ such that the diagram below commutes.
\end{itemize}
$$
\begin{CD}
Q_n @>g>> Q_n\\
@V\pi_K V V @VV\pi_L V\\
(Q_n)_K @>> \varphi > (Q_n)_L
\end{CD}
$$
\end{proposition}

\begin{proof}
(i) Suppose that $\varphi:(Q_n)_K\to (Q_n)_L$ is a graph isomorphism. 
By Theorem \ref{class. dist.}, $(Q_n)_K$ is a rectagraph of valency $n$, so  $(Q_n)_L$ is a rectagraph of valency $n$, in which case Theorem \ref{class. dist.} implies that $d_L\geq 5$. Hence $\pi_K$ and $\pi_L$  are coverings by  Lemma \ref{covering}.
Let $\theta:=\pi_L\varphi^{-1}$. Now $\theta: Q_n\to (Q_n)_K$ is a covering,  so there exists $y\in VQ_n$ such that $y\theta= 0^K$, and $\theta$ induces a bijection from $Q_n(y)$ onto $(Q_n)_K(0^K)=\{e_i^K:1\leq i\leq n\}$ by Lemma \ref{nbd2}.
 There exists a
 covering $g:Q_n\to Q_n$ for which $0^g=y$ and $e_i^g=e_i^K\theta^{-1}$ for all $ i\leq n$ by \cite[Lemma 3.1]{BamDevFawPra2015}. Now
 $\pi_K$ and $g\theta$ agree on $\{0\}\cup Q_n(0)$,  so $\pi_K=g\theta$ by \cite[Lemma 3.2]{BamDevFawPra2015}. Thus  $\pi_K \varphi=g\pi_L$, and the diagram commutes. Since $g$ is surjective, it must be injective, so $g\in \Aut(Q_n)$. By \cite[Lemma 2.4]{BamDevFawPra2015}, $K=K^{\pi_K}$ and $L=L^{\pi_L}$. Hence for $k\in K$,  $$ (g^{-1}kg)\pi_L=g^{-1}k\pi_K\varphi=g^{-1}\pi_K\varphi=g^{-1}g\pi_L=\pi_L,$$
 so $g^{-1}Kg\leq L$; since $2^n/|K|=2^n/|L|$, we conclude that $g^{-1}Kg=L$, as desired.
 
 (ii) Suppose that $L=g^{-1}Kg$ for some $g\in \Aut(Q_n)$. Define $\varphi: (Q_n)_K\to (Q_n)_L$  by $x^K\mapsto (x^g)^L$ for all $x\in VQ_n$. Clearly the diagram commutes, and  $\varphi$ is a graph isomorphism.
\end{proof}

We remark that Proposition \ref{conjugate}(i) cannot be improved to include $d_K\leq 4$. Observe that if $K, L\leq \Aut(Q_n)$ such that 
   $(Q_n)_K\simeq (Q_n)_L$ but $K$ and $L$ are not conjugate in $\Aut(Q_n)$, then $(d_K,d_L)\in \{0,1,2\}\times\{0,1,2\}$ or $ \{(3,3),(4,4)\}$  by Lemmas \ref{cycle},  \ref{covering} and \ref{a c}(i).  
   When $n=6$, examples of such $K$ and $L$ exist for each possible $(d_K,d_L)$ by \cite{Magma}.

Proposition \ref{conjugate} implies that rectagraphs arising from binary linear codes are fundamentally different to those arising from  subgroups of $\Aut(Q_n)$ not contained in $\mathbb{F}_2^n$, for if $C$ is a binary linear code in $\mathbb{F}_2^n$, then $g^{-1}Cg\leq \mathbb{F}_2^n$ for all $g\in \Aut(Q_n)$.

\begin{proof}[Proof of Theorem \ref{conjuate simple}]
This follows from Proposition \ref{conjugate} and \cite[Lemma 2.4]{BamDevFawPra2015}.
\end{proof}

We finish this section with some examples. Recall that if $d_K=\infty$, then $(Q_n)_K$ is the $n$-cube, and if $d_K$ is finite, then $d_K\leq n$. The groups with largest possible finite minimum distance are described in the following.

\begin{example}
\label{large}
Let $K\leq \Aut(Q_n)$ where $n\geq 4$. If $d_K=n-1$ or $n$, then $K=\{0,x\}$ where $x\in VQ_n$ and $|x|=n-1$ or $n$ respectively. In the latter case, $(Q_n)_K$ is the folded $n$-cube.  In both cases, $K$ is a binary linear code, so $(Q_n)_K$ is vertex-transitive.
\end{example}

Next we consider the minimum distance of subgroups of $\Aut(Q_n)$ of order $2$.

\begin{example}
\label{K2}
Let $K\leq \Aut(Q_n)$ where $|K|=2$. Let $(x,\sigma)$ be the involution in $K$. Write $x=(x_1,\ldots,x_n)$, and let $\fix(\sigma)$ be the set of fixed points of $\sigma$. Then $d_K=|\{i\in \fix(\sigma):x_i=1\} |$.
\end{example}

In Example \ref{large}, we saw that the groups $K\leq \Aut(Q_n)$ with largest possible minimum distance are binary linear codes, and so the corresponding graphs are vertex-transitive. 
Now we see that there exist $K\leq \Aut(Q_n)$ such that $d_K$ is large but $(Q_n)_K$ is not vertex-transitive.

\begin{example}
\label{not vt}
Suppose that  $n\geq 8$. Choose $x=(x_1,\ldots,x_n)\in \mathbb{F}_2^n$ with $\wt{x}=n-1$ or $n$, and let $\sigma:=(i\ j)$ where  $x_i=1=x_j$. Let $K:=\{0,(x,\sigma)\}$. Then $K\leq \Aut(Q_n)$ and $d_K=n-3$ or $n-2$ respectively. Let $N:=N_{\Aut(Q_n)}(K)=\{(y,\tau)\in \Aut(Q_n):y^\sigma=y, x^\tau=x, \sigma\tau=\tau\sigma\}$. Then $e_{i}^K\notin (0^K)^N$, so $(Q_n)_K$ is not vertex-transitive by Theorem \ref{conjuate simple}.
\end{example}

\section{Distance 2 graph of $(Q_n)_K$}
\label{s: halved}

It is well known  that for  a binary linear code $C$ with $d_C\geq 2$, the graph $(Q_n)_C$ is bipartite precisely when  $C\leq E_n$. Moreover, when $C$ is not even, the bipartite double $(Q_n)_C.2$ is isomorphic to $(Q_n)_{C\cap E_n}$, and for $d_C\geq 4$, every halved  graph of $(Q_n)_C.2$ is isomorphic to the distance $2$ graph of $(Q_n)_C$. These results generalise to   subgroups $K$ of $\Aut(Q_n)$ with $d_K\geq 2$. Note that if $K$ is such a group and $L:=K\cap (E_n: S_n)$, then $L$ is an even subgroup of $\Aut(Q_n)$ with   $d_L\geq 2$. 

\begin{lemma}
\label{even}
Let  $K\leq \Aut(Q_n)$  where $d_K\geq 2$, and let $\Pi:=(Q_n)_K$. 
\begin{itemize}
\item[(i)] 
 $\Pi$ is bipartite with parts  $\{x^K\in V\Pi:\wt{x}\equiv i \mod 2\}$ for $i=0,1$  if and only if $K$ is even.
 \item[(ii)]  If $K$ is not even, then $ \Pi.2\simeq (Q_n)_{L}$ where $L:=K\cap (E_n: S_n)$.
  \item[(iii)] If $K$ is not even  and $d_K\geq 4$, then every halved graph of $\Pi.2$ is isomorphic to $\Pi_2$.
\end{itemize}
\end{lemma}

\begin{proof}
(i) Let $B_i:=\{x^K\in V\Pi:\wt{x}\equiv i \mod 2\}$ for $i=0,1$. Suppose that $K$ is even. If $x\in VQ_n$ and $k\in K$, then $\wt{x}\equiv \wt{x^k}\mod 2$, so $B_0$ and $B_1$ partition $V\Pi$. If $x^K$ and $y^K$ are adjacent and $x^K\in B_0$, then  $y^K\in B_1$ by Lemma \ref{nbd}. Thus $\Pi$ is bipartite with parts $B_0$ and $B_1$.

Now suppose that $K$ is not even.    Let $r:=\min\{d_{Q_n}(x,x^k): x\in VQ_n,\ k\in K\setminus (E_n: S_n)\}$. Choose $x\in VQ_n$ and $k\in K\setminus (E_n: S_n)$ such that $d_{Q_n}(x,x^k)=r$. There exists a path $(x_0,x_1,\ldots,x_r)$ in $Q_n$ such that $x_0=x$ and $x_r=x^k$.  If $x_i^K=x_j^K$ for some $0\leq i<j<r$, then $x_i^\ell=x_j$ for some $\ell\in K$, so $d(x_i,x_i^\ell)=j-i<r$. Thus $\ell\in E_n: S_n$.  Since $d_{Q_n}(x,x_i)=i$ and $d_{Q_n}(x_i,x^{k\ell^{-1}})=d_{Q_n}(x_i^\ell, x^k)=r-j$, it follows that $d_{Q_n}(x,x^{k\ell^{-1}})\leq r+i-j<r$. But $k\ell^{-1}\in K\setminus (E_n: S_n)$, a contradiction.  Since $r$ is odd and $r\geq d_K\geq 2$, it follows that $(x_0^K,\ldots,x_{r}^K)$ is an odd cycle, so $\Pi$ is not bipartite.

(ii) Define a map $\varphi:(Q_n)_L\to \Pi.2$ by $x^L\mapsto (x^K,\wt{x}\mod 2)$ for all $x\in VQ_n$.  Using Lemma \ref{nbd2}, it is routine to verify that $\varphi$ is a graph isomorphism.

(iii) A halved graph $\Gamma$ of $\Pi.2$ has vertex set $\{(x^K,i):x\in VQ_n\}$ for some $i\in \mathbb{F}_2$, so there is a bijection $\varphi:\Gamma\to \Pi_2$ defined by $(x^K,i)\mapsto x^K$ for all $x\in VQ_n$.  Since $\Pi$ has no triangles by Lemma \ref{a c}, $(x^K,i)$ is adjacent to $(y^K,i)$ in $\Gamma$ if and only if $d_\Pi(x^K,y^K)=2$, so $\varphi$ is a graph isomorphism. 
\end{proof}

It is also well known that for a binary linear code $C$ with $d_C\geq 2$, the halved graphs of $(Q_n)_C$ are isomorphic. However, as we will see shortly, this is not always true for arbitrary subgroups $K$ of $\Aut(Q_n)$ with $d_K\geq 2$. By Lemma \ref{even}(ii) and (iii),  it is true when $K=M\cap (E_n:S_n)$ for some  non-even  subgroup $M$ of $\Aut(Q_n)$ with $d_M\geq 4$. Of course, such a group $M$ normalises $K$, and we can generalise this observation  as follows.

\begin{proposition}
\label{halved}
Let  $K\leq \Aut(Q_n)$ be even  where $d_K\geq 2$. If   $N_{\Aut(Q_n)}(K)$ is not even,   then the halved graphs of $(Q_n)_K$ are isomorphic. 
\end{proposition}

\begin{proof}
 Let $\Gamma$ and $\Sigma$ be the halved graphs of $(Q_n)_K$, and let $g\in N_{\Aut(Q_n)}(K)$ be  such that $g\notin E_n: S_n$. By Proposition \ref{conjugate}(ii), the image of $g$ in $N_{\Aut(Q_n)}(K)/K$ is an automorphism of $(Q_n)_K$, and  $V\Gamma^g=V\Sigma$ by Lemma \ref{even}(i), so $\Gamma$ and $\Sigma$ are isomorphic.
\end{proof}

Translating by the vector $(1,\ldots,1)$ is central in $\Aut(Q_n)$, so for odd $n$, the group  $N_{\Aut(Q_n)}(K)$ is not even, and we obtain the following immediate consequence of Proposition \ref{halved}.

\begin{corollary}
\label{odd iso}
Let  $K\leq \Aut(Q_n)$ be even where $d_K\geq 2$. If $n$ is odd, then the halved graphs of $(Q_n)_K$ are isomorphic.
\end{corollary}

Next we give an example of a normal quotient whose halved graphs are not isomorphic. 

\begin{example}
\label{exp:halved}
Let $x:=e_{1,2,3,4}$ and $y:=e_{1,3,6,8}$. Let $\sigma:=(15)(26)(37)(48)$ and $\tau:=(12)(34)(56)(78)$. Let $K:=\langle (x,\sigma),(y,\tau)\rangle\leq \Aut(Q_8)$. Note that $K$ is isomorphic to the quaternion group. Now 
$K$ is even and $d_K=4$, so   $\Pi:=(Q_n)_K$ is bipartite.  But $|\Pi_2(0^K)|=13$ while $|\Pi_2(e_1^K)|=14$,  
so the halved graphs of $\Pi$ are not isomorphic by Lemma \ref{even}(i).
\end{example}

If $K$ is an even subgroup of $\Aut(Q_n)$  such that $d_K\geq 5$, then the halved graphs of $(Q_n)_K$ are regular of valency $\tbinom{n}{2}$ by Lemmas \ref{a c} and \ref{counting}, in which case the failure we observed in Example \ref{exp:halved} cannot occur. It would be interesting to determine in general when the halved graphs of a bipartite normal quotient are isomorphic, especially for those subgroups with minimum distance at least $5$. We can at least say the following.

\begin{example}
Let $K\leq \Aut(Q_n)$ where $K$ is even, $|K|=2$ and $d_K\geq 2$. Let $(x,\sigma)$ be the involution in $K$. Then $i^\sigma=i$ for some $i\in [n]$ by Example \ref{K2}, so $e_i\in N_{\Aut(Q_n)}(K)$. Thus the halved graphs of $(Q_n)_K$ are  isomorphic  by Proposition \ref{halved}.
\end{example}

 \section{Locally $T_n$ graphs}
 \label{s: proof}

In this section, we prove Theorems \ref{main even} and \ref{main aut}. First we give a generalisation of one direction of Theorem \ref{main even}.

\begin{lemma}
\label{loc Tn}
Let  $K\leq \Aut(Q_n)$ and $\Pi:=(Q_n)_K$ where $n\geq 2$.
 If $d_K\geq 7$, then any connected component of $\Pi_2$ is locally $T_n$.
\end{lemma}

\begin{proof}
 Let $\Gamma$ be a connected component of $\Pi_2$, and let $x^K\in V\Gamma$. By Lemma \ref{nbd2}, there is a surjective map $\varphi:T_n\to [\Gamma(x^K)]$ defined by $\{i,j\}\mapsto (x+e_{i,j})^K$ for all $\{i,j\}\in VT_n$, and this map is injective by Lemma \ref{trick}. Let $\{i,j\}\in VT_n$.  By Lemma \ref{nbd2}, $$\Gamma((x+e_{i,j})^K)=\{(x+e_{i,j}+e_{\ell,m})^K\in V\Pi: \{\ell,m\}\in  VT_n\}.$$ 
 If $(x+e_{i',j'})^K=(x+e_{i,j}+e_{\ell,m})^K$ for some $\{i',j'\},\{m,n\}\in  VT_n$ where  $|\{i,j\}\cap \{\ell,m\}|=0$, then $6\geq \wt{e_{i',j'}+e_{i,j}+e_{\ell,m}}\geq d_K$ by Lemma \ref{trick}, a contradiction. Thus 
 $$\Gamma(x^K)\cap\Gamma((x+e_{i,j})^K)=\{(x+e_{\ell,m})^K\in V\Pi:\{\ell,m\}\in  VT_n,\ |\{i,j\}\cap \{\ell,m\}|=1 \},$$
 and we conclude that $\varphi$ is a graph isomorphism.
 \end{proof}

\begin{proof}[Proof of Theorem \ref{main even}]
Suppose that $\Gamma$ is connected and locally $T_n$. Then $\Gamma$ is a halved graph of some bipartite rectagraph $\Pi$ of valency $n$ with $c_3=3$ by \cite[Proposition 4.3.9]{BroCohNeu1989}, so we may apply Corollary \ref{main rect}. Also, $K$ is even by Lemma \ref{even}(i).

Conversely, suppose that $\Gamma$ is a halved graph of $(Q_n)_K$ for some  $K\leq \Aut(Q_n)$ where $K$ is even and  $d_K\geq 7$. Then $\Gamma$ is  locally $T_n$ by Lemma \ref{loc Tn}. The group $K$ is unique up to conjugacy in $\Aut(Q_n)$ by Theorem \ref{conjuate simple}.
\end{proof}

\begin{proof}[Proof of Theorem \ref{main aut}]
 Let  $\Pi:=(Q_n)_K$  and $N:=N_{\Aut(Q_n)}(K)$. Now $\Aut(\Pi)=N/K$ by Theorem  \ref{conjuate simple}.
 Since  $\Gamma$ is locally $T_n$ and $n\geq 5$,   $\Aut(\Gamma)$ is  the setwise stabiliser of $V\Gamma$ in $\Aut(\Pi)$ by \cite[Lemma 4.2]{BamDevFawPra2015}.  Recall that $V\Gamma=\{x^K\in V\Pi:\wt{x}\equiv i\mod 2\}$ where $i=0$ or $1$ by Lemma \ref{even}(i).  Thus $\Aut(\Gamma)=N_{E_n:S_n}(K)/K$.
\end{proof}

Hence the automorphism group of a connected locally $T_n$ graph is essentially known for  $n\geq 5$. This group can also be determined for $n\leq 4$. Suppose that $\Gamma$ is a connected locally $T_n$ graph  where $n\leq 4$. Then $\Gamma$ is locally  $K_1$ for $n=2$, $K_3$ for $n=3$, or $K_{3[2]}$ for $n=4$, so $\Gamma$ is  $K_2$, $K_4$ or $K_{4[2]}$ respectively (i.e., 
 $\Gamma$ is the halved $n$-cube). Hence $\Aut(\Gamma)$ is $S_2$, $S_4$ or $S_2\wr S_4$ respectively. Note that $\Aut(\Gamma)$ is only isomorphic to $N_{E_n: S_n}(K)/K$ when $n=3$.

We finish by observing that there are examples of locally $T_n$ graphs that are not vertex-transitive.

\begin{example}
\label{lt not vt}
 Let $K$ and $N$ be as in Example \ref{not vt} where $n$ is chosen so that  $(Q_n)_K$ is bipartite and $d_K\geq 7$. Let $\Gamma$ be the halved graph of $(Q_n)_K$ with vertex set $\{x^K\in V(Q_n)_K:\wt{x}\ \mbox{is even}\}$.  If 
$\ell\in [n]\setminus\{i,j\}$, then $e_{i,\ell}^K\notin (0^K)^N$, so $\Gamma$ is not vertex-transitive by Theorem \ref{main aut}.\end{example}

\bibliographystyle{acm}
\bibliography{jbf_references}

\end{document}